\documentclass[10pt,a4paper,twoside,final]{amsart}

\usepackage[latin1]{inputenc}
\usepackage[T1]{fontenc}
\usepackage{amssymb}
\usepackage[UKenglish]{babel}

\numberwithin{equation}{section}

\newtheorem{theorem}{Theorem}
\newtheorem{lemma}[theorem]{Lemma}

\theoremstyle{remark}
\newtheorem{remark}[theorem]{Remark}
\newtheorem*{ackn}{Acknowledgement}

\newcommand{\abs}[1]{\lvert#1\rvert}
\newcommand{\biggabs}[1]{\biggl\lvert#1\biggr\rvert}
\newcommand{\norm}[1]{\lVert#1\rVert}
\newcommand{\bignorm}[1]{\bigl\lVert#1\bigr\rVert}

\newcommand{\BMOA}{\mathit{BMOA}}
\newcommand{\VMOA}{\mathit{VMOA}}
\newcommand{\Bloch}{\mathcal{B}}
\newcommand{\LMOA}{\mathit{LMOA}}
\newcommand{\LVMOA}{\mathit{LMOA}_0}

\newcommand{\D}{\mathbb{D}}
\newcommand{\T}{\mathbb{T}}
\newcommand{\C}{\mathbb{C}}

\DeclareMathOperator{\dist}{dist}

\title[Essential norms of integration operators]
{Weak compactness and essential norms\\of integration
operators}

\author[Laitila]{Jussi Laitila}
\address{ISER, University of Essex, Colchester CO4 3SQ, United Kingdom}
\email{jlaitila@essex.ac.uk}

\author[Miihkinen]{Santeri Miihkinen}
\address{Department of Mathematics and Statistics, University of
Helsinki, Box~68, FI-00014 Helsinki, Finland}
\email{santeri.miihkinen@helsinki.fi}

\author[Nieminen]{Pekka J.\ Nieminen}
\address{Department of Mathematics and Statistics, University of
Helsinki, Box~68, FI-00014 Helsinki, Finland}
\email{pjniemin@cc.helsinki.fi}

\subjclass[2010]{Primary 47B38; Secondary 30H10, 30H35, 47B07.}

\date{21 January 2011}

\thanks{The second author was supported by the Academy of Finland,
project 134757.}

\begin{document}

\begin{abstract}
Let $g$ be an analytic function on the unit disc and consider
the integration operator of the form
$T_g f(z) = \int_0^z fg'\,d\zeta$. We show that on
the spaces $H^1$ and $\BMOA$ the operator $T_g$ is
weakly compact if and only if it is compact. In the case of
$\BMOA$ this answers a question of Siskakis and Zhao.
More generally, we estimate the essential and weak essential
norms of $T_g$ on $H^p$ and $\BMOA$.
\end{abstract}

\maketitle

\section{Introduction}

Let $\D$ be the open unit disc in the complex plane $\C$ and
$g\colon \D \to \C$ an analytic function. We consider the
generalized Volterra integration operator $T_g$ defined by
\[
   T_gf(z) = \int_0^z f(\zeta)g'(\zeta)\,d\zeta, \qquad z \in \D,
\]
for functions $f$ analytic in $\D$. As special cases this
includes the classical Volterra operator for $g(z) = z$
and the Cesàro operator for $g(z) = -\log(1-z)$.

In the general form such operators were first introduced by
Pommerenke~\cite{Pommerenke:schfaf} to study exponentials of
$\BMOA$ functions. He observed, in particular, that $T_g$ is
bounded on the Hardy space $H^2$
if and only if $g$ belongs to $\BMOA$, the space of
analytic functions with bounded mean oscillation on the
unit circle. A detailed study of the operators $T_g$
was later initiated by Aleman and Siskakis \cite{Aleman:intoh},
who showed that
Pommerenke's boundedness characterization is valid on each
$H^p$ for $1 \leq p < \infty$ and that $T_g$ is
compact on $H^p$ if and only if $g \in \VMOA$.

Subsequently a number of authors have extended this line of research to
a variety of other spaces and contexts; we refer the reader to the
surveys \cite{Aleman:somopc,Siskakis:volosa} for more information
and further references. In particular,
in \cite{Siskakis:voltos} Siskakis and Zhao considered
$T_g$ as an operator acting on $\BMOA$ and
characterized its boundedness and compactness in terms
of logaritmically weighted $\BMOA$ and $\VMOA$ conditions
placed on $g$ (see below for precise statements).

The main purpose of this paper is to address the
\emph{weak compactness} of $T_g$ on $H^1$ and
$\BMOA$. We will namely show that on each of these spaces
$T_g$ is weakly compact precisely when it is compact.
In the setting of $\BMOA$ this result provides a negative
answer to a question posed by Siskakis and Zhao in
\cite[Sec.~3]{Siskakis:voltos}.
More generally, we will derive estimates for the essential and
weak essential
norms of $T_g$ on $H^p$ and $\BMOA$, extending an earlier
result of Rättyä~\cite{Rattya} for the $H^2$ case.
These results are contained in Theorems \ref{thm:Hardy}
and~\ref{thm_bmoa} below.

Recall here that a linear operator $T$ on a Banach space is
weakly compact if it maps the unit ball of the space into a set
whose closure is compact in the weak topology.
The essential and weak essential norms of $T$, denoted
by $\norm{T}_e$ and $\norm{T}_w$, are the distances of $T$
(in the operator norm) from the closed ideals of compact and weakly
compact operators, respectively.

As usual, we define $\BMOA$
as the space of analytic functions $g\colon \D \to \C$ such that
\begin{equation} \label{eq:semi}
   \norm{g}_* = \sup_{a\in\D}\norm{g\circ\sigma_a-g(a)}_{H^2} < \infty,
\end{equation}
where $\norm{\ }_{H^2}$ is the standard norm of $H^2$ and
$\sigma_a(z) = (a-z)/(1-\bar{a}z)$ is the conformal automorphism of
$\D$ that interchanges $0$ and $a$. Equivalently, $g \in H^2$ and
the boundary values
of $g$ have bounded mean oscillation on the unit circle.
Introducing the norm $\abs{g(0)} + \norm{g}_*$ makes $\BMOA$ a
Banach space. Its closed subspace $\VMOA$ consists of
those $g$ for which
$\norm{g\circ\sigma_a-g(a)}_{H^2} \to 0$ as
$\abs{a} \to 1$. For detailed accounts on $\BMOA$ and $\VMOA$
we refer the reader to \cite{Baernstein,Garnett:bouaf,Girela:anafbm}.

Throughout the paper we use the notation $A \lesssim B$
to indicate that $A \le cB$ for some positive constant $c$ whose
value may change from one occurrence into another and which
may depend on $p$.
If $A \lesssim B$ and $B \lesssim A$, we say that the quantities
$A$ and $B$ are equivalent and write $A \simeq B$.

\begin{theorem} \label{thm:Hardy}
Let $g \in \BMOA$. Then, for $1 \leq p < \infty$,
\[
   \norm{T_g\colon H^p \to H^p}_e \simeq \dist(g,\VMOA),
\]
and
\[
   \norm{T_g\colon H^1 \to H^1}_w \simeq \dist(g,\VMOA).
\]
In particular, $T_g$ is weakly compact on $H^1$
if and only if it is compact, or equivalently, $g \in \VMOA$.
\end{theorem}

The \emph{logarithmic $\BMOA$ space}, denoted here by $\LMOA$, consists
of those analytic functions $g\colon \D \to \C$ for which
\begin{equation} \label{eq:logsemi}
   \norm{g}_{*,\log} = \sup_{a\in\D}\lambda(a)
      \norm{g\circ\sigma_a-g(a)}_{H^2} < \infty,
\end{equation}
where $\lambda(a)=\log (2/(1-\abs{a}))$.
It is a Banach space under the norm
$\abs{g(0)} + \norm{g}_{*,\log}$.
The \emph{logarithmic $\VMOA$ space}, denoted by
$\LVMOA$, is defined by the corresponding ``little-oh'' condition.
Note that $\LMOA \subset \VMOA$.

Siskakis and Zhao \cite{Siskakis:voltos} proved
that $T_g$ is bounded (resp.\ compact) on $\BMOA$, or equivalently
on $\VMOA$, if and only if $g \in \LMOA$ (resp.\ $g \in \LVMOA$).
Alternative proofs can be found in \cite{MacCluer:vanlcm,Zhao:logcm}.
We extend these results as follows.

\begin{theorem}\label{thm_bmoa}
Let $g\in \LMOA$. Then $T_g\colon \BMOA \to \BMOA$ satisfies
\[
   \norm{T_g}_e \simeq \norm{T_g}_w \simeq \dist(g,\LVMOA).
\]
In particular, $T_g$ is weakly compact on $\BMOA$ if and
only if it is compact, or equivalently, $g \in \LVMOA$.
The same estimates hold for the restriction
$T_g\colon \VMOA \to \VMOA$.
\end{theorem}

Observe that the distances in Theorems \ref{thm:Hardy}
and~\ref{thm_bmoa} can be calculated in terms of the respective
seminorms \eqref{eq:semi} and \eqref{eq:logsemi} because
the constant functions belong to $\VMOA$ and $\LVMOA$.
There are various function-theoretic formulas for estimating such
distances and we collect a pair of these in Section~\ref{sec:Pre}.
The proof of Theorem~\ref{thm:Hardy} is then given in
Section~\ref{sec:Hardy} and the proof of Theorem~\ref{thm_bmoa}
in Section~\ref{sec:BMOA}.

\begin{remark}
The weak compactness of generalized Volterra operators has
previously been considered
in the setting of the little Bloch space $\Bloch_0$; see e.g.\
Hu~\cite{Hu:extcob}. Recall, however, that on $\Bloch_0$,
as well as on the Bergman space $A^1$, all weakly compact operators
are compact because $A^1$ is the dual of $\Bloch_0$ and (being
isomorphic to $\ell^1$) it has the Schur property (see e.g.\
\cite[Chap.\ 4 and~5]{Zhu}). Moreover,
since $T_g$ acting on the Bloch space $\Bloch$ can be viewed as
the biadjoint of its restriction to $\Bloch_0$, it follows that
the weak compactness of $T_g$ is equivalent to
its compactness even on $\Bloch$.

Note, on the contrary, that since $H^1$ and $\BMOA$ are non-reflexive
spaces containing complemented copies of the Hilbert space $\ell^2$
(see e.g.\ \cite[Sec.~9]{Girela:anafbm}), they admit bounded operators
that are not weakly compact and weakly compact operators that are not
compact.
\end{remark}

\section{Distance formulas}
\label{sec:Pre}

There exist various function-theoretic quantities for estimating the
distance of a general $\BMOA$ function from $\VMOA$.
Typically they involve a limsup version of an expression
defining (or equivalent to) the $\BMOA$ norm (see
e.g.\ \cite{CarmonaCufi,StegengaStephenson}). One version
is furnished by Lemma~\ref{le:dist} below. For completeness we
briefly sketch its proof, especially because the arguments in
\cite{CarmonaCufi,StegengaStephenson} do not seem to be
directly adaptable to exponents in the scale $0 < p < 1$,
which will be important to us in Section~\ref{sec:Hardy}.

Recall here that every function $g \in \BMOA$ satisfies a ``reverse
Hölder inequality'' which implies that for each $0 < p < \infty$,
\begin{equation}\label{eq:RH}
   \norm{g}_* \simeq \sup_{a\in\D} \norm{g\circ\sigma_a-g(a)}_{H^p},
\end{equation}
where the proportionality constants depend on $p$.
Likewise, $g \in \VMOA$ if and only if
$\norm{g\circ\sigma_a-g(a)}_{H^p} \to 0$ as $\abs{a} \to 1$.
(See e.g.\ \cite[Corollary 3]{Baernstein}.)

\begin{lemma} \label{le:dist}
For $g \in \BMOA$ and $0 < p < \infty$,
\[
   \dist(g,\VMOA) \simeq
   \limsup_{\abs{a}\to 1} \norm{g\circ\sigma_a - g(a)}_{H^p}.
\]
\end{lemma}

\begin{proof}
The lower estimate for $\dist(g,\VMOA)$ is an easy consequence of
\eqref{eq:RH} and the corresponding characterization of $\VMOA$
functions.

To prove the upper estimate, one approximates $g$ by the
$\VMOA$ functions $g_r(z) = g(rz)$ for $0 < r < 1$.
Fix $0 < \eta < 1$. It is easy to check that
$g_r\circ\sigma_a - g_r(a)$ converges to
$g\circ\sigma_a - g(a)$ in $H^p$ uniformly for $\abs{a} \leq \eta$
as $r \to 1$.
Hence, by \eqref{eq:RH},
\[ \begin{split}
   \dist(g,\VMOA)
   &\leq \limsup_{r\to 1} \norm{g-g_r}_*  \\
   &\simeq \limsup_{r\to 1} \sup_{\abs{a}>\eta}
    \bignorm{[g\circ\sigma_a-g(a)] - [g_r\circ\sigma_a-g_r(a)]}_{H^p} \\
   &\leq \sup_{\abs{a}>\eta} \norm{g\circ\sigma_a-g(a)}_{H^p} +
    \sup_{r>\eta}\sup_{\abs{a}>\eta}
    \norm{g_r\circ\sigma_a-g_r(a)}_{H^p}.
\end{split} \]
We may write
\begin{equation} \label{eq:subord}
   g_r\circ\sigma_a - g_r(a)
   = [g\circ\sigma_{ra} - g(ra)] \circ\psi_{r,a},
\end{equation}
where $\psi_{r,a}=\sigma_{ra}\circ r\sigma_a$ is an analytic
self-map of $\D$ that fixes the origin. Therefore the
Littlewood subordination theorem (see e.g.\ \cite[Thm 1.7]{Duren})
yields $\norm{g_r\circ\sigma_a-g_r(a)}_{H^p}
\leq \norm{g\circ\sigma_{ra}-g(ra)}_{H^p}$. The upper
estimate follows as $\eta \to 1$.
\end{proof}

In the case of logarithmic $\BMOA$ we have the following
analogue. 

\begin{lemma} \label{le:logdist}
For $g \in \LMOA$,
\[
   \dist(g,\LVMOA) \simeq \limsup_{\abs{a}\to 1}
      \lambda(a)\norm{g\circ\sigma_a - g(a)}_{H^2}.
\]
\end{lemma}

\begin{proof}
The lower estimate is obtained by an application
of the triangle inequality and the definition of $\LVMOA$ as in the
proof of Lemma~\ref{le:dist}.

The proof of the upper estimate follows the previous idea as well,
but requires a refined version of the subordination argument.
Indeed, we again approximate $g$ by the functions $g_r(z)=g(rz)$
(belonging to $\LVMOA$ by \cite[Lemma 3.5]{Siskakis:voltos}) to get
the estimate
\begin{equation} \label{eq:distest} \begin{split}
   \dist(g,\LVMOA)
   \leq &\sup_{\abs{a}>\eta}
        \lambda(a)\norm{g\circ\sigma_a-g(a)}_{H^2} \\
        & + \sup_{r>\eta}\sup_{\abs{a}>\eta}
        \lambda(a)\norm{g_r\circ\sigma_a-g_r(a)}_{H^2}.
\end{split} \end{equation}
Applying the weighted subordination principle
of \cite[Prop.~2.3]{Laitila:weicob} to \eqref{eq:subord}, we obtain
\[
   \norm{g_r\circ\sigma_a - g_r(a)}_{H^2}
   \lesssim \norm{\psi_{r,a}}_{H^2}
            \norm{g\circ\sigma_{ra} - g(ra)}_{H^2}.
\]
A calculation shows that $\norm{\psi_{r,a}}_{H^2}^2 =
r^2(1-|a|^2)/(1-r^4|a|^2)$, so, for $r$ and $\abs{a}$ sufficiently
close to $1$, we have
\[ \begin{split}
   \lambda(a)\norm{g_r\circ\sigma_a - g_r(a)}_{H^2}
   &\lesssim \frac{\lambda(a)\sqrt{1-|a|^2}}{\sqrt{1-|ra|^2}}
      \norm{g\circ\sigma_{ra} - g(ra)}_{H^2}  \\
   &\le \lambda(ra) \norm{g\circ\sigma_{ra} - g(ra)}_{H^2}
\end{split} \]
because $\lambda(s)\sqrt{1-s^2}$ is decreasing on, say,
$\mathopen[\frac{9}{10},1\mathclose)$. Combining this with
\eqref{eq:distest} and
letting $\eta \to 1$ yields the required estimate.
\end{proof}

\section{Proof of Theorem~\ref{thm:Hardy}}
\label{sec:Hardy}

We will make use of the standard test functions in $H^p$,
$1\leq p < \infty$, defined by
\begin{equation} \label{eq:testf}
   f_a(z) = \biggl[ \frac{1-\abs{a}^2}{(1-\bar{a}z)^2} \biggr]^{1/p}
\end{equation}
for each $a \in \D$. Note that $\norm{f_a}_{H^p} = 1$.
For $p=2$ this is just the normalized
reproducing kernel of $H^2$.
According to a theorem of Aleman and Cima \cite[Thm~3]{AlemanCima},
there exists a constant $c_{p,q} > 0$ such that
\begin{equation} \label{eq:AlCi}
   \norm{T_g f_a}_{H^p} \geq c_{p,q}\norm{g\circ\sigma_a-g(a)}_{H^q}
\end{equation}
whenever $0 < q < p/2$ (for example, $q = p/4$).

In order to deal with the weak essential
norm of $T_g$ on $H^1$ a localization argument is needed.
Let $m$ be the normalized Lebesgue measure on the unit circle
$\T = \partial\D$.
We will utilize the classical Dunford--Pettis criterion
(see e.g.\ \cite[Thm 5.2.9]{AK}) which says that
a set $F \subset L^1(m)$ is relatively compact in the weak topology
of $L^1(m)$ if and only if it is uniformly integrable, i.e.\
\[
   \sup_{f \in F} \int_E \abs{f}\,dm \to 0
   \quad\text{as $m(E) \to 0$.}
\]
The application of this criterion in our setting
is based on the following lemma.

\begin{lemma} \label{le:Arc}
For non-zero $a \in \D$, let $I(a) =
\bigl\{e^{i\theta} : \abs{\theta-\arg a} < (1-\abs{a})^{1/6} \bigr\}$
and $f_a(z) = (1-\abs{a}^2)\big/(1-\bar{a}z)^2$.
Then
\[
   \lim_{\abs{a} \to 1} \int_{\T\setminus I(a)} \abs{T_gf_a}\,dm = 0.
\]
\end{lemma}

\begin{proof}
We may assume $0 < a < 1$ (by rotation-invariance) and
$g(0)=0$. It is easy to check that
$\abs{1 - are^{i\theta}} \geq c\abs{\theta}$ for all
$0 \leq r < 1$ and $\abs{\theta} \leq \pi$, where $c > 0$ is
an absolute constant. Thus, for $0 \leq r < 1$ and
$(1-\abs{a})^{1/6} \leq \abs{\theta} \leq \pi$, we have the
uniform estimates
\begin{align*}
   \abs{f_a(re^{i\theta})}
      &\lesssim \frac{1 - a}{\abs{1 - are^{i\theta}}^2}
      \lesssim \frac{1 - a}{\abs{\theta}^2} \leq (1 - a)^{2/3},
\\
   \abs{f_a'(re^{i\theta})}
      &\lesssim \frac{1 - a}{\abs{1 - are^{i\theta}}^3}
      \lesssim \frac{1 - a}{\abs{\theta}^3} \leq (1 - a)^{1/2}.
\end{align*}
The functions $g$ and $T_gf_a$ have radial limits at almost every
point of $\T$. Therefore, for a.e.\ $\zeta \in \T \setminus I(a)$,
we may use integration by parts and the above estimates to get
\[ \begin{split}
   \abs{T_gf_a(\zeta)}
   &= \biggabs{\int_0^1 f_a(r\zeta) g'(r\zeta) \zeta \,dr}
   \le \abs{f_a(\zeta)g(\zeta)} +
       \int_0^1 \abs{f_a'(r\zeta)g(r\zeta)} \,dr \\
   &\lesssim (1 - a)^{2/3} \abs{g(\zeta)} +
       (1 - a)^{1/2} \int_0^1 \abs{g(r\zeta)} \,dr. 
\end{split} \]
Since $\BMOA$ is contained in the Bloch space and
consequently $g$ has at most logarithmic growth, the last
integral here is bounded by a constant multiple of $\norm{g}_*$.
Hence
\[
   \int_{\T\setminus I(a)} \abs{T_g f_a} \,dm
   \lesssim (1 - a)^{2/3} \norm{g}_{H^1} + (1 - a)^{1/2} \norm{g}_*.
\]
This yields the required result.
\end{proof}

\begin{proof}[Proof of Theorem~\ref{thm:Hardy}]
For every $1 \leq p < \infty$, the upper estimate for
$\norm{T_g}_e$ follows easily from the linearity
of $T_g$ with respect to $g$. Indeed, for each $h \in \VMOA$,
the operator $T_h$ is compact and hence
$\norm{T_g}_e \leq \norm{T_g-T_h} = \norm{T_{g-h}}
\simeq \norm{g-h}_*$ by \cite[Thm~1]{Aleman:intoh}.

To establish the lower estimates, we first consider the case
$1 < p < \infty$. Define functions $f_a$ by \eqref{eq:testf}.
Since $f_a \to 0$ weakly in $H^p$ as $\abs{a} \to 1$,
for every compact operator $K$ on $H^p$ we have
$\norm{Kf_a}_{H^p} \to 0$. Consequently
\[
   \norm{T_g}_e
   \geq \limsup_{\abs{a}\to 1} \norm{T_g f_a}_{H^p}
   \gtrsim \limsup_{\abs{a}\to 1}
      \norm{g\circ\sigma_a - g(a)}_{H^{p/4}},
\]
where the last estimate follows from \eqref{eq:AlCi}.
But the right-hand side here is equivalent to
$\dist(g,\VMOA)$ by Lemma~\ref{le:dist}.

We finally consider the case $p=1$ and derive the lower bound
for $\norm{T_g}_w$ on $H^1$. Let $S$ be an arbitrary weakly
compact operator on $H^1$ and,
as before, consider the test functions
$f_a(z) = (1-\abs{a}^2)\big/(1-\bar{a}z)^2$ for $a \in \D$.
Since $\norm{f_a}_{H^1} = 1$, we have the estimate
\[
   \norm{T_g}_w \geq \norm{(T_g-S)f_a}_{H^1}
      \geq \int_{I(a)} \abs{T_gf_a}\,dm -
           \int_{I(a)} \abs{Sf_a}\,dm,
\]
where $I(a)$ is the arc of Lemma~\ref{le:Arc}.
Since the set $\{Sf_a : a\in\D\}$ is relatively weakly compact
in $H^1$ and hence uniformly integrable in $L^1(m)$, the last
integral on the right-hand side tends to zero as $\abs{a} \to 1$.
Hence Lemma~\ref{le:Arc} yields
$\norm{T_g}_w \geq \limsup_{\abs{a}\to 1} \norm{T_gf_a}_{H^1}$.
The argument is then finished as above. 
\end{proof}

\section{Proof of Theorem \ref{thm_bmoa}}
\label{sec:BMOA}

We start by recalling
that $\BMOA$ functions admit a characterization in terms of
Carleson measures (see e.g.\ \cite[VI.3]{Garnett:bouaf}).
For any arc $I \subset \T$, write $\abs{I} = m(I)$
and let $S(I)=\{z\in\D : 1-\abs{z} < |I|,\: z/\abs{z}\in I\}$ denote
the Carleson window determined by $I$. Given an analytic function
$g\colon \D \to \C$, define
\[
   \mu(g,I) = \int_{S(I)} \abs{g'(z)}^2(1-\abs{z}^2)\,dA(z),
\]
where $A$ denotes the normalized Lebesgue area measure on $\D$. Then
\begin{equation}\label{eq_bmoa_char}
   \norm{g}_* \simeq \sup_{I\subset\T}
      \biggl(\frac{\mu(g,I)}{\abs{I}}\biggr)^{1/2}
\end{equation}
with the understanding that $g \in \BMOA$ if and only if
the right-hand side is finite. Also,
\begin{align}\label{eq_vmoa_char}
   g \in \VMOA \quad\Leftrightarrow\quad
   \lim_{\abs{I}\to 0}\frac{\mu(g,I)}{\abs{I}}=0.
\end{align}
Furthermore, for logarithmic $\BMOA$ we have the following equivalence
that is contained in the proof of \cite[Lemma 3.4]{Siskakis:voltos}:
\begin{equation} \label{eq:LMOAequiv}
   \limsup_{\abs{a}\to 1}
   \lambda(a)\norm{g\circ\sigma_a - g(a)}_{H^2}
   \simeq \limsup_{\abs{I}\to 0}
      \biggl(\log\frac{2}{\abs{I}}\biggr)
      \biggl(\frac{\mu(g,I)}{\abs{I}}\biggr)^{1/2}.
\end{equation}
In view of Lemma~\ref{le:logdist}, this gives another estimate
for the distance of a function $g \in \LMOA$ from $\LVMOA$.

A key tool in the proof of Theorem~\ref{thm_bmoa} is an idea of
Le{\u\i}bov~\cite{Leibov:subvs} on how to construct isomorphic
copies of the sequence space $c_0$ inside $\VMOA$.
As usual, here $c_0$ denotes the Banach space of all complex
sequences converging to zero equipped with the supremum norm.
The following reformulation of Le{\u\i}bov's result
is taken from~\cite{Laitila:comwcc}.

\begin{lemma}[{\cite[Prop.~6]{Laitila:comwcc}}]\label{lemma_leibov}
Let $(f_n)$ be a sequence in $\VMOA$ such that $\norm{f_n}_* \simeq 1$
and $\norm{f_n}_{H^2}\to 0$ as $n\to \infty$. Then there is a
subsequence $(f_{n_j})$ which is equivalent to the natural basis of
$c_0$; that is, the map
$\iota\colon (\lambda_j) \to \sum_j \lambda_j f_{n_j}$ is an
isomorphism from $c_0$ into $\VMOA$.
\end{lemma}

The utility of this lemma lies in the fact that
$c_0$ has the Dunford--Pettis property:\
every weakly compact linear operator from $c_0$
into any Banach space maps weak-null sequences into
norm-null sequences (see e.g.\ \cite[Sec.~5.4]{AK}).

After these preparations we are ready to carry out the proof
of Theorem~\ref{thm_bmoa}.

\begin{proof}[Proof of Theorem \ref{thm_bmoa}]
We first consider the case of $T_g$ acting on $\BMOA$.
Recall that $\norm{T_g}_w \le \norm{T_g}_e$.

To derive the upper estimate for $\norm{T_g}_e$, we argue as in
the proof of Theorem~\ref{thm:Hardy}:\ for each $h \in \LVMOA$,
the operator $T_h$ is compact on $\BMOA$ and hence
$\norm{T_g}_e \leq \norm{T_g-T_h} = \norm{T_{g-h}}
\simeq \norm{g-h}_{*,\log}$ by
Theorem~3.1 and Lemma~3.4 of \cite{Siskakis:voltos}.

To prove the lower estimate for $\norm{T_g}_w$, we first choose a
sequence $(I_n)_{n=1}^\infty$ of subarcs of $\T$ such that
$\abs{I_n} \to 0$ and
\[
   \lim_{n\to \infty} \biggl(\log\frac{2}{\abs{I_n}}\biggr)^2
   \frac{\mu(g,I_n)}{\abs{I_n}} =
   \limsup_{\abs{I}\to 0}\biggl(\log\frac{2}{\abs{I}}\biggr)^2
   \frac{\mu(g,I)}{\abs{I}} \equiv \alpha.
\]
In view of Lemma~\ref{le:logdist}
and equivalence \eqref{eq:LMOAequiv} it is enough to show that
$\norm{T_g}_w \gtrsim \sqrt{\alpha}$.
Note that $\alpha$ is finite because $g \in \LMOA$.

For $n \geq 1$, let $u_n=(1-\abs{I_n})\xi_n$ where $\xi_n$ is the
midpoint of $I_n$.
By passing to a subsequence if necessary, we may assume that
$(u_n)$ is a Cauchy sequence. Define $f_n(z) = \log(1-\overline{u_n}z)$
for $z \in \D$. A calculation shows that
\[
   \abs{f_n(z)} \ge c\log\frac{2}{\abs{I_n}},\qquad
   z \in S(I_n),
\]
for all $n \geq 1$ and a uniform constant $c>0$.
Hence
\begin{equation}\label{eq_estA2}
   \frac{\mu(T_g f_n,I_n)}{\abs{I_n}} =
   \frac{1}{\abs{I_n}}
   \int_{S(I_n)} \abs{f_n(z)}^2\abs{g'(z)}^2 (1-\abs{z}^2)\,dA(z)
   \ge \frac{c\alpha}{2^2}
\end{equation}
for $n$ larger than some (henceforth fixed) $N \geq 1$.

It is known that $(f_n)$ is a bounded sequence in $\BMOA$.
In fact, since $f_n$ extends continuously to $\overline\D$, we have
$f_n\in\VMOA$ and consequently $T_g f_n \in \VMOA$ because
$T_g$ maps $\VMOA$ into itself. Hence, by
applying \eqref{eq_vmoa_char} and passing to a further subsequence
if necessary, we may assume that
\begin{equation}\label{eq_estB2}
   \frac{\mu(T_g f_n,I_{n+1})}{\abs{I_{n+1}}} \le \frac{c\alpha}{4^2}
\end{equation}
for all $n \geq 1$.

Let $h_n=f_{n+1}-f_n$. Then, by combining
\eqref{eq_estA2} and \eqref{eq_estB2} and
applying the triangle inequality, we get, for $n \geq N$,
\begin{equation}\label{eq_conseq1}
   \frac{c\alpha}{4^2}
   \le \frac{\mu(T_g (f_{n+1}-f_n),I_{n+1})}{\abs{I_{n+1}}}
   \le C\norm{T_g h_n}_*^2,
\end{equation}
where $C> 0$ is a constant that stems from
\eqref{eq_bmoa_char}. Thus
$\norm{h_n}_*^2 \ge c\alpha/16C\norm{T_g} > 0$.
On the other hand,
\[
   \norm{h_n}_{H^2}^2 =
   \sum_{k=1}^{\infty}\frac{\abs{u_n^k-u_{n+1}^k}^2}{k^2}\to 0
\]
as $n\to \infty$. Therefore, by Lemma~\ref{lemma_leibov},
there is a subsequence $(h_{n_j})$ such that the map
$\iota\colon (\lambda_j)\mapsto \sum_{j=1}^{\infty}\lambda_j h_{n_j}$
is an isomorphism from $c_0$ into $\BMOA$.

Let now $S$ be any weakly compact operator on $\BMOA$. Then
$S\circ\iota$ is weakly compact from $c_0$ to $\BMOA$ and
since the standard unit vector basis $(e_j)$ of $c_0$ converges to
zero weakly in $c_0$, the Dunford--Pettis property of $c_0$ implies
$\norm{S h_{n_j}}_* = \norm{(S \circ \iota)e_j}_* \to 0$
as $j\to\infty$. Since $(h_n)$ is bounded in $\BMOA$, we have,
by \eqref{eq_conseq1},
\[
   \norm{T_g - S} \gtrsim \norm{T_g h_{n_j} - S h_{n_j}}_*
   \ge \tfrac{1}{4}\sqrt{\!c\alpha/C} - \norm{S h_{n_j}}_*.
\]
This yields that
$\norm{T_g-S} \gtrsim \sqrt{\alpha}$ as $j\to\infty$.
Hence $\norm{T_g}_w \gtrsim \sqrt{\alpha}$ and the proof of the
lower estimate is complete.

Finally consider $T_g$ as an operator on $\VMOA$.
The upper estimate for $\norm{T_g}_e$ is obtained
exactly as in the $\BMOA$ case because the compact approximants
$T_h$, $h \in \LVMOA$, take $\VMOA$ into itself.
Moreover, since the test functions $f_n$ and $h_n$ above
belong to $\VMOA$, our argument for the lower estimate
of $\norm{T_g}_w$ works in the $\VMOA$ case as well.  This finishes
the proof of Theorem~\ref{thm_bmoa}.
\end{proof}

\begin{ackn}
Part of this research was carried out when the second and third
authors were visiting Lund University in Sweden in 2010. They
acknowledge the hospitality of the Centre for Mathematical Sciences,
and thank Alexandru Aleman and Karl-Mikael Perfekt
for useful remarks.
\end{ackn}


\bibliographystyle{amsplain}

\begin{thebibliography}{99}

\bibitem{AK}
F.~Albiac and N.~Kalton,
\emph{Topics in Banach Space Theory}, Springer, New York, 2006.

\bibitem{Aleman:somopc}
A.~Aleman,
\emph{A class of integral operators on spaces of analytic functions},
Topics in complex analysis and operator theory, 3--30,
Univ.\ M\'alaga, M\'alaga, 2007.

\bibitem{AlemanCima}
A.~Aleman and J.A.~Cima,
\emph{An integral operator on $H^p$ and Hardy's inequality},
J.~Analyse Math.\ 85 (2001), 157--176.

\bibitem{Aleman:intoh}
A.~Aleman and A.G.~Siskakis,
\emph{An integral operator on $H^p$},
Complex Variables Theory Appl.\ 28 (1995), no.~2, 149--158.

\bibitem{Baernstein}
A.~Baernstein II,
\emph{Analytic functions of bounded mean oscillation}, Aspects of
contemporary complex analysis (Durham, 1979), Academic Press,
London, 1980, pp.\ 3--36.

\bibitem{CarmonaCufi}
J.J.~Carmona and J.~Cufí,
\emph{On the distance of an analytic function to VMO},
J.~London Math.\ Soc.\ (2) 34 (1986), no.~1, 52--66.

\bibitem{Duren}
P.L.~Duren, \emph{Theory of $H^p$ Spaces}, Academic Press, New York,
1970; reprinted by Dover, New York, 2000.

\bibitem{Garnett:bouaf}
J.B.~Garnett, \emph{Bounded Analytic Functions},
rev.\ ed., Springer, New York, 2007.

\bibitem{Girela:anafbm}
D.~Girela,
\emph{Analytic functions of bounded mean oscillation},
Complex function spaces (Mekrij\"arvi, 1999),
Univ.\ Joensuu Dept.\ Math.\ Rep.\ Ser., 4, Univ.\ Joensuu, 2001,
pp.\ 61--170.

\bibitem{Hu:extcob}
Z.~Hu,
\emph{Extended Ces{\`a}ro operators on the Bloch space in the
unit ball of $C^n$}, Acta Math.\ Sci.\ Ser.\ B Engl.\ Ed.\ 23
(2003), no.~4, 561--566.

\bibitem{Laitila:comwcc}
J.~Laitila, P.J.~Nieminen, E.~Saksman and H.-O.~Tylli,
\emph{Compact and weakly compact composition operators on BMOA},
Complex Anal.\ Oper.\ Theory (to appear).
Preprint version in arXiv:0912.3487v2 [math.FA].

\bibitem{Laitila:weicob}
J.~Laitila, \emph{Weighted composition operators on BMOA},
Comput.\ Methods Funct.\ Theory 9 (2009), no.~1, 27--46.

\bibitem{Leibov:subvs}
M.V.~Le{\u\i}bov,
\emph{Subspaces of the VMO space} (Russian), Teor.\
Funktsi\u\i\ Funktsional.\ Anal.\ i Prilozhen.\ 46 (1986), 51--54;
English transl.\ in J.~Soviet Math.\ 48 (1990), no.~5, 536--538.

\bibitem{MacCluer:vanlcm}
B.~MacCluer and R.~Zhao,
\emph{Vanishing logarithmic Carleson measures},
Illinois J.\ Math.\ 46 (2002), no.~2, 507--518.

\bibitem{Pommerenke:schfaf}
Ch.~Pommerenke,
\emph{Schlichte Funktionen und analytische Funktionen von
beschr\"ankter mittlerer Oszillation},
Comment.\ Math.\ Helv.\ 52 (1977), no.~4, 591--602.

\bibitem{Rattya}
J.~Rättyä, \emph{Integration operator acting on Hardy and weighted
Bergman spaces}, Bull.\ Austral.\ Math.\ Soc.\ 75 (2007), no.~3,
431--446.

\bibitem{Siskakis:volosa}
A.G.~Siskakis,
\emph{Volterra operators on spaces of analytic functions---a survey},
Proceedings of the First Advanced Course in Operator Theory and
Complex Analysis, Univ.\ Sevilla Secr.\ Publ., Seville, 2006,
pp.~51--68.

\bibitem{Siskakis:voltos}
A.G.~Siskakis and R.~Zhao,
\emph{A Volterra type operator on spaces of analytic functions},
Function spaces (Edwardsville, 1998), Contemp.\ Math.\ 232 (1999),
299--311.

\bibitem{StegengaStephenson}
D.A.~Stegenga and K.~Stephenson,
\emph{Sharp geometric estimates of the distance to VMOA},
The Madison Symposium on Complex Analysis (Madison, 1991),
Contemp.\ Math.\ 137 (1992), 421--432.

\bibitem{Zhao:logcm}
R.~Zhao, \emph{On logarithmic Carleson measures},
Acta Sci.\ Math.\ (Szeged) 69 (2003), 605--618.

\bibitem{Zhu}
K.~Zhu,
\textit{Operator Theory in Function Spaces}, Dekker, New York,
1990; 2nd ed.\ by Amer.\ Math.\ Soc., Providence, 2007.

\end{thebibliography}

\end{document}